\documentclass{article}
\usepackage{graphicx} 
\usepackage{hyperref}
\usepackage{amsmath}
\usepackage{amssymb}
\usepackage{amsthm}
\usepackage{tikz}
\usetikzlibrary{cd}
\usepackage{biblatex}
\usepackage{geometry}

\usepackage{mathpazo}
\usepackage[libertine,cmintegrals,cmbraces,vvarbb]{newtxmath}
\DeclareMathAlphabet{\mathcal}{OMS}{cmsy}{m}{n}

\title{Coherent Conditions: Algebraic Geometry for Arbitrary Classes of Algebras}
\author{K. R. van Nispen}

\theoremstyle{plain}
\newtheorem{theorem}{Theorem}[section]
\newtheorem{lemma}[theorem]{Lemma}
\newtheorem{corollary}[theorem]{Corollary}
\newtheorem{proposition}[theorem]{Proposition}

\theoremstyle{definition}
\newtheorem{definition}{Definition}[section]

\newtheorem{remark}[theorem]{Remark}

\newcommand{\Spec}{\operatorname{Spec}}
\newcommand{\RSpec}{\operatorname{RSpec}}
\newcommand{\nil}{\operatorname{nil}}
\newcommand{\rad}{\operatorname{rad}}

\begin{document}

\maketitle

\begin{abstract}
    Universal algebraic geometry is generalised from solutions of equations in a single algebra to the study of $\varphi$- or $K$-spectra, akin to the prime spectrum of a ring.
    We explore their basic properties and constructions, give a correspondence between certain quantifier-free propositions and closed sets in the Zariski topology of a free algebra, and show the connection with current UAG.
    Lastly, equationally Noetherian classes and irreducible spectra are explored.
\end{abstract}

\tableofcontents

\pagebreak

\section*{Introduction}

Classical algebraic geometry is the study of algebraic sets over some algebraically closed field and their corresponding coordinate algebras.
It found its counterparts in other areas of math, when people began studying the algebraic geometry of groups (some of the first being \cite{GroupsWithParExp}, \cite{OneVarEqinFreeGroups} and \cite{VerbalTopology}, more of these covered in \cite{daniyarova2011algebraicgeometryalgebraicstructures}), for example, or Lie algebras (perhaps starting with \cite{ZeroDivInAmal}, and continued in \cite{Daniyarova_2005}, \cite{Daniyarova_2006}, \cite{2Daniyarova_2006}).
As these did not have as many nice properties as fields, like the ability to perform localisation, or have the union of two algebraic sets again be an algebraic set, there naturally arise different, and often more logic-focused, questions.\par
Very recently, in 2002, B. Plotkin submitted his paper \textit{Seven Lectures on the Universal Algebraic Geometry} \cite{plotkin2002sevenlecturesuniversalalgebraic}, laying the groundwork for the universalisation of these (classical) algebraic geometries into the domain of universal algebra.
It, among other things, showed the fundamental connection between algebraic geometry and algebraic logic, an idea which has persisted and carried on into this paper.
Later, in 2008, the paper \textit{Unification theorems in algebraic geometry} by E. Daniyarova, Alexei G. Myasnikov and V. Remeslennikov \cite{Daniyarova2008UnificationTI} was submitted, unifying a number of theorems about so-called equationally Noetherian algebras.
This was followed in 2010 by \cite{daniyarova2011algebraicgeometryalgebraicstructures}, further expanding on the unification theorems and exhibiting the categorical duality between algebraic sets and coordinate algebras.
A number of papers followed, also by E. Daniyarova, Alexei G. Myasnikov and V. Remeslennikov, exploring the universal algebraic geometry over a single algebra in depth.\par
In this paper we aim to generalise this notion of algebraic geometry to arbitrary classes of algebras of the same type, similarly to the prime spectrum of a ring, called coherent conditions.
Unfortunately we lose the convenience of our algebraic sets being `affine' - subsets of $A^n$ for some algebra $\mathbf{A}$ - but in section 3.2 it will be shown briefly that, in the case where the class of algebras $K$ consists of a single algebra, all the properties one cares about are carried over equivalently.\par
Section 2 will be devoted to introducing coherent conditions, and exploring their elementary properties. The first, lattice-theoretic part of the generalised Nullstellensatz will be stated.
In Section 3, we consider the spectra of $V$-free algebras, and state the fundamental connections between coherent conditions and equational logic, one of note being the second, logic part of the generalised Nullstellensatz.
In Section 4, we explore equationally Noetherian classes of algebras, and loosely generalize Unification Theorem C in \cite{daniyarova2011algebraicgeometryalgebraicstructures}.
Lastly, in Section 5, we cover irreducible closed sets, and give a classification of reduced algebras with an irreducible spectrum, similar to that in \cite{daniyarova2011algebraicgeometryalgebraicstructures} or \cite{Daniyarova2008UnificationTI}.

\section{Preliminaries}

\subsection{Universal algebra}

In this paper we will mainly use the notation and conventions of \cite{burris}, besides the fact that we conventionally include $0$ in the natural numbers.

A type, or signature $\mathscr{F}$ is a set of operation symbols $f$ each with an associated natural number $n$ called its \textit{arity}. $\mathscr{F}_n$ denotes the operation symbols in $\mathscr{F}$ of arity $n$.
An algebra $\mathbf{A}$ of type $\mathscr{F}$ is a tuple,
$$\mathbf{A} = \langle A; (f^\mathbf{A})_{f\in \mathscr{F}}\rangle$$
such that, if $f \in \mathscr{F}_n$, then $f^\mathbf{A} \colon A^n \rightarrow A$ is an $n$-ary operation on $A$.
The $f^\mathbf{A}$ are called the fundamental operations of $\mathbf{A}$, and $A$ the \textit{universe} or underlying set of $\mathbf{A}$.
The set of homomorphisms between $\mathbf{A}$ and $\mathbf{B}$ is denoted $\hom(\mathbf{A}, \mathbf{B})$, and the lattice of congruences of $\mathbf{A}$, $\operatorname{Con}\mathbf{A}$.\par
The set of terms in the type $\mathscr{F}$ with variables in $X$ is written as $T_\mathscr{F}(X)$, or simply $T(X)$ if the signature is known from context.
This defines the absolutely free algebra $\mathbf{T}(X)$.
$\text{At}_\mathscr{F}(X)$ denotes the set of atomic formulas of $\mathscr{F}$ with variables in $X$; the equational identities $p\approx q$ for $p, q \in T(X)$.\par
If $T \subseteq \operatorname{At}(X)$ is a set of equations, then for an algebra $\mathbf{A}$ and $\vec{a}  \in A^X$ we write $\mathbf{A} \models T(\vec{a})$ if $\vec{a}$ satisfies all equations in $T$ in $\mathbf{A}$.
In other words, $T$ can be used as a shorthand for the proposition $$\bigwedge_{p\approx q \in T}p\approx q$$
Formulas of the form $T \rightarrow t\approx s$ are called \textit{preidentities}.
A \textit{quasi-identity} is a preidentity where $T$ is finite.\par

We distinguish the following operators for a class of algebras $K$:
\begin{enumerate}
    \item $IS(K)$, the class of algebras embeddable into some member of $K$.
    \item $P(K)$, the class of algebras which are some arbitrary product of members of $K$.
    \item $ISP(K)$, the least prevariety containing $K$; the class of algebras embeddable into some product of members in $K$.
    Like \cite{daniyarova2011algebraicgeometryalgebraicstructures}, we assume the convention of a product indexed by the empty set to be the trivial algebra.
    \item $Q(K)$, the class of algebras satisfying the same quasi-identities as $K$; the least quasivariety containing $K$.
    \item $V(K) = HSP(K)$, the least variety containing $K$.
    \item $K_\omega$, the finitely generated algebras in $K$.
    \item $\textbf{Res}(K)$ the class of algebras separated by $K$
    \item $\textbf{Sep}_\omega(K)$ the class of algebras $n$-separated by $K$ for all $n$. 
    \item $\mathbf{Dis}(K)$, the class of algebras discriminated by $K$
\end{enumerate}

The definitions of a class seperated by or discriminated by $K$ are already given in, for example, \cite{daniyarova2011algebraicgeometryalgebraicstructures}.
However, the notion of $n$-separated is one we introduce ourselves.

\begin{definition}
    A homomorphism $h \colon \mathbf{A} \rightarrow \mathbf{B}$ \textit{separates} two elements from $A$ if they are unequal in the image, or equivalently lie in different blocks in the kernel of $h$.\par
    An algebra $\mathbf{A}$ is separated by $K$ if for any $a_0\neq a_1 \in A$ there is a homomorphism into some member of $K$ separating $a_0$ and $a_1$.
    An algebra is $n$-separated by $K$ if for any finite family of separated pairs $a_i\neq b_i$ there is a homomorphism $h$ into $K$ with $h(a_i)\neq h(b_i)$ for all $i$.\par
    Lastly, an algebra is discriminated by $K$ if it is locally embeddable into $K$; for any finite $W\subseteq A$ there is a homomorphism into $K$ such that its restriction onto $W$ is injective, i.e. pairwise separates the elements from $W$.
\end{definition}

For a class $K$, the free algebra generated by $X$ is denoted $\mathbf{F}_K(X)$, and is defined as the quotient $T(X)/\psi$ where $\psi = \{ \langle p, q\rangle \in T(X)^2 \mid K\models p\approx q\}$.
Elements of $\mathbf{F}_K(X)$ will freely be identified by a term representing the corresponding  equivalence class.
Furthermore, we will treat subsets $T \subseteq F_K(X)^2$ as systems of equations, as, for any two choices of representatives $T_1', T_2' \subseteq \text{At}(X)$ and class $L$ contained in $V(K)$ we have $L \models T'_1 \iff L \models T'_2$.
This means that we can also treat $T$ as a (possibly infinitary) proposition on algebras in $V(K)$.

\subsection{Closure operators}

A closure operator on a set $X$ is a function $C \colon \mathcal{P}(X)\rightarrow \mathcal{P}(X)$, $\mathcal{P}(X)$ being the powerset of $X$, such that for $A, B\subseteq X$:
\begin{itemize}
    \item $A \subseteq C(A)$
    \item $A\subseteq B \implies C(A)  \subseteq C(B)$
    \item $C(C(A)) = C(A)$
\end{itemize}
The pair $\langle X, C\rangle$ is then a closure system, which we will often refer to as simply $X$.\par
The set of closed sets, $\operatorname{cl}(C)$ (consequently $\operatorname{cl}(X)$), is the set of $C(A)$ for $A \subseteq X$.
These are closed under arbitrary intersection and, in fact, closure operators are in 1-1 correspondence with subsets of $\mathcal{P}(X)$ closed under arbitrary intersection.
For such a collection $M \subseteq \mathcal{P}(X)$, the corresponding closure operator is then given by
$$C_M(A) = \bigcap \{ S \in M \mid A\subseteq S\}$$

Every closure system $\langle X, C\rangle$ yields a complete lattice structure on $\operatorname{cl}(C)$.
The arbitrary meet is, naturally, given by intersection, and for $(A_i)_{i \in I}$ a collection of closed sets $\bigvee_{i\in I} A_i := C(\bigcup_{i \in I} A_i)$.\par

A morphism of closure systems $f \colon \langle X, C\rangle \rightarrow \langle Y, D\rangle$ is a function $f\colon X\rightarrow Y$ where the preimage of a closed set of $Y$ is again closed in $X$.
The composition of two morphisms is again a morphism, and this yields the category of closure operators.\par
We then have a nonstandard definition, which shall be very useful in our case.

\begin{definition}
    A quasi-isomorphism is a morphism of closure systems $f \colon X \rightarrow Y$ such that $f$ maps closed sets of $X$ to closed sets of $Y$, and $f^{-1} \colon \operatorname{cl}(D) \rightarrow \operatorname{cl}(C)$ is a bijection.
\end{definition}

\begin{proposition} \label{prop:qiisnice}
    Let $f \colon X\rightarrow Y$ be a quasi-isomorphism.
    Denote the mapping $B \mapsto f^{-1}(B)$ for closed sets $B \subseteq Y$ with $F$.
    Then:
    \begin{enumerate}
        \item The mapping $G : A \mapsto f(A)$ for closed sets $A\subseteq X$ is the inverse of $F$.
        \item $G$ preserves arbitrary intersections.
    \end{enumerate}
\end{proposition}
\begin{proof}
    To begin with $(1)$, note that $f(f^{-1}(A)) = A$ for all $A\subseteq Y$, so we have, for $A \subseteq Y$ closed:
    $$G(F(A)) = f(f^{-1}(A)) = A$$
    Thus $G\circ F = id$.
    As $F$ is a bijection, and $G$ is a well-defined function, it must also be a bijection, thus being the inverse of $F$.\par
    For $(2)$, let $(A_i)_{i  \in I}$ be a family of closed sets of $X$.
    Then notice:
    \begin{align*}
        G(\bigcap_{i \in I} A_i) &= G(\bigcap_{i \in I} F(G(A_i)))\\
        &= G(F(\bigcap_{i \in I} G(A_i)))\\
        &= \bigcap_{i \in I} G(A_i)
    \end{align*}
    and we are done.
\end{proof}

It will be apparent that many properties of a closure system (even those which are not properties of the complete lattice of closed sets) are equivalent to those same properties of a quasi-isomorphic closure system.

In general the closed sets of a closure operator do not form a topology on the underlying set.
This is only if the closure operator is topological; for all $A, B \subseteq X$:
$$C(A\cup B) = C(A)\vee C(B) = C(A)\cup C(B)$$
\begin{remark}
    As both the direct image and preimage maps ($G$ and $F$ in the above proposition) preserve unions, $X$ is topological iff $Y$ is, whenever there is a quasi-isomorphism between them.
\end{remark}
To deal with this, we introduce the concept of topologization.

\begin{definition}
    If $X$ is a closure system, then $X^\text{top}$ denotes the topological closure system with prebasis $\operatorname{cl}(X)$, called the topologization of $X$.
    The sets which are closed in $X^\text{top}$ are called \textit{topologically closed} in $X$.
\end{definition}

The reader can verify that the topologically closed sets are always of the form
$$\bigcap_{i \in I} A_{i, 1} \cup \dots\cup A_{i, k_i}$$
for $A_{i, k}$ closed in $X$, and $k_i$ is potentially $0$ for the closed set $\varnothing$.

\begin{proposition} \label{prop:topisendfunc}
    If $f\colon X\rightarrow Y$ is a morphism of closure systems, then the same function
    $\hat{f} \colon X^\text{top} \rightarrow Y^\text{top}$ is continous (also a morphism of closure systems).
\end{proposition}
\begin{proof}
    Straightforward from the fact that $f^{-1}$ preserves unions and intersections.
\end{proof}

\begin{corollary}
    An isomorphism of closure systems yields a homeomorphism of their topologisations.
\end{corollary}

\pagebreak

\section{Coherent conditions}

Unless otherwise specified, in this section we fix a variety $V$ over the type $\mathscr{F}$.

\subsection{$\varphi$-spectra, and fundamental classes}

The study of modern algebraic geometry relies for a big part on the following two properties:
\begin{enumerate}
    \item Let $f\colon R\rightarrow S$ be a ring homomorphism, and $\mathfrak{p}\subseteq S$ a prime ideal, then $f^{-1}(\mathfrak{p})$ is also prime.
    In other words: $f^{-1}(\Spec S) \subseteq \Spec R$
    \item Let $\mathfrak{p}\subseteq R$ be a prime ideal, and $I\subseteq\mathfrak{p}$, then $\mathfrak{p}/I$ is a prime ideal of $R/I$.
\end{enumerate}
One of the consequences is that $\Spec$ is a contravariant functor from the category of rings to that of topological spaces.
Another is that $I$ is prime iff $R/I$ is an integral domain.
We aim to capture these properties, and their surprisingly far-reaching consequences, with coherent conditions and their associated spectra.

\begin{definition}
    A coherent condition $\varphi$ on $V$ assigns to every algebra $\mathbf{A} \in V$ its $\varphi$-\textit{spectrum} $\Spec^\varphi\mathbf{A}\subseteq \operatorname{Con}\mathbf{A}$ of congruences.
    These spectra have to satisfy the following two conditions:
    \begin{enumerate}
        \item Let $h\colon \mathbf{A} \rightarrow \mathbf{B}$ be a homomorphism.
        Then $h^{-1}(\Spec^\varphi\mathbf{B}) \subseteq \Spec^\varphi\mathbf{A}$
        \item If $\psi \in \Spec^\varphi\mathbf{A}$, and $\theta \subseteq\psi$, then $\psi/\theta \in \Spec^\varphi\mathbf{A}/\theta$
    \end{enumerate}
    Congruences in $\Spec^\varphi\mathbf{A}$ are the $\varphi$-\textit{distinguished congruences} of $\mathbf{A}$.
\end{definition}

An immediate, but important consequence of the definition is the following proposition.

\begin{proposition} \label{prop:natprojisemb}
    Let $\pi_\theta \colon \mathbf{A}\rightarrow \mathbf{A}/\theta$ be the natural projection.
    Then the preimage
    $$\pi_\theta^{-1} \colon \Spec^\varphi\mathbf{A}/\theta \rightarrow \Spec^\varphi\mathbf{A}$$
    is an injection, the image of which being $\{ \psi \in \Spec^\varphi\mathbf{A} \mid \theta \subseteq \psi \}$
\end{proposition}
\begin{proof}
    The fact that $\pi_\theta^{-1}$ must be an injection follows from the correspondence theorem.
    Then for the image, take $\psi \in \Spec^\varphi\mathbf{A}$ with $\theta\subseteq\psi$, and notice that $\psi/\theta$ is $\varphi$-distinguished, and $\pi_\theta^{-1}(\psi/\theta) = \psi$.
\end{proof}

Of course, the canonical example of a coherent condition is that of prime ideals on the variety of commutative rings.
As one can check, the set of radical ideals of a commutative ring also forms a coherent condition.
Another important example of a coherent condition is that of maximal congruences in the variety of Boolean algebras, as this leads in part to Stone duality.
This can, however. be seen as a special case of prime ideals on commutative rings, due to the equivalence between Boolean rings and Boolean algebras. \cite{burris}\par
As the reader might have noticed, the above examples of coherent conditions are all defined as sets of congruences where the quotient algebra is a member of some particular class. 
For prime and radical ideals these are respectively the integral domains and reduced rings, and for Boolean algebras the class simply consists of the two-element Boolean algebra.
We can generalise this construction.

\begin{definition}
    Let $K$ be a class of algebras contained in $V$.
    Then $\varphi := \Phi_V(K)$ assigns to every algebra $\mathbf{A} \in K$ the following set of congruences:
    $$\Spec^\varphi\mathbf{A} := \{ \theta \in \operatorname{Con}\mathbf{A} \mid \mathbf{A}/\theta \in IS(K) \}$$
\end{definition}

In other words, the $\Phi_V(K)$-distinguished congruences are precisely the kernels of homomorphisms into $K$.
We will also call $\Spec^\varphi\mathbf{A}$ the $K$-spectrum of $\mathbf{A}$.

\begin{proposition}
    $\varphi := \Phi_V(K)$ is a coherent condition on $V$.
\end{proposition}
\begin{proof}
    To prove that $\varphi$ must satisfy the first property, consider a homomorphism $h \colon \mathbf{A} \rightarrow \mathbf{B}$, and $\theta \in \Spec^\varphi\mathbf{B}$, so $\mathbf{B}/\theta$ embeds into some algebra $\mathbf{C} \in K$.
    Then consider the following diagram:
    \begin{center}
        \begin{tikzcd}
            \mathbf{A} \arrow[d] \arrow[r, "h"] & \mathbf{B} \arrow[d]\\
            \mathbf{A}/h^{-1}(\theta) \arrow[r, dashed] & \mathbf{B}/\theta \arrow[r, tail] & \mathbf{C}
        \end{tikzcd}
    \end{center}
    Where the vertical maps are the natural projections, and the dashed arrow is given by the map
    $$x/h^{-1}(\theta) \mapsto h(x)/\theta$$
    which the reader can check is well-defined.
    Furthermore, it is injective basically by definition of the preimage.
    Thus, we get the chain
    $$\mathbf{A}/h^{-1}(\theta) \hookrightarrow \mathbf{B}/\theta \hookrightarrow \mathbf{C}$$
    $$\implies \mathbf{A}/h^{-1}(\theta) \in IS(K)$$
    meaning that $h^{-1}(\Spec^\varphi\mathbf{B}) \subseteq\Spec^\varphi\mathbf{A}$.\par
    For the second property, suppose $\psi, \theta$ congruences of $\mathbf{A}$ with $\psi$ being $\varphi$-distinguished.
    By the third isomorphism theorem:
    $$\implies (\mathbf{A}/\theta)/(\psi/\theta) \cong\mathbf{A}/\psi \in IS(K)$$
    Hence $\psi/\theta \in \Spec^\varphi\mathbf{A}/\theta$.
\end{proof}

As we see, this already gives us a huge class of coherent conditions, one uniquely for each class of algebras closed under subalgebras and isomorphisms, in fact.
Surprisingly, these are \textit{the only} coherent conditions which exist.
In other words, every coherent condition is $\Phi_V(K)$ for some class $K$.
To show this we introduce the fundamental class of a coherent condition.

\begin{definition}
    Let $\varphi$ be a coherent condition on $V$.
    The \textit{fundamental class} of $\varphi$, denoted $F(\varphi)$, is the class of algebras $\mathbf{A} \in V$ such that $\Delta_A \in \Spec^\varphi\mathbf{A}$, so the minimal congruence (diagonal) is $\varphi$-distinguished.
\end{definition}

For example, for the coherent condition of prime ideals on commutative rings, its fundamental class is the class of integral domains.

\begin{proposition}
    The fundamental class is closed under isomorphisms and subalgebras.
\end{proposition}
\begin{proof}
    By the first isomorphism theorem, this is equivalent to the fundamental class being closed under embeddings.
    As such, let $\varphi$ be a coherent condition and $\mathbf{A} \in F(\varphi)$, and consider an embedding $i \colon \mathbf{B} \rightarrow \mathbf{A}$.
    By definition, $\Delta_A$ is $\varphi$-distinguished.
    $$\implies \Delta_B = i^{-1}(\Delta_A) \in \Spec^\varphi\mathbf{A}$$
    $$\implies \mathbf{B} \in F(\varphi)$$
    and we are done.
\end{proof}

There is another, also fairly useful, characterisation of the fundamental class of a coherent condition.

\begin{lemma} \label{lem:altfundclassdef}
    Let $\varphi$ be a coherent condition on $V$.
    Then $F(\varphi)$ is also all the algebras isomorphic to $\mathbf{A}/\theta$ for some $\mathbf{A} \in V$ and $\theta \in \Spec^\varphi\mathbf{A}$. 
\end{lemma}
\begin{proof}
    Straightforward from Proposition \ref{prop:natprojisemb}: $\theta\in \operatorname{Con}\mathbf{A}$ is $\varphi$-distinguished iff $\Delta_{A/\theta} = \theta/\theta$ is.
\end{proof}

We are now ready to prove the correspondence between coherent conditions on $V$ and classes closed under embeddings.

\begin{theorem} \label{theo:coherency}
    For any coherent condition $\varphi$, it is the case that $\Phi_V(F(\varphi)) = \varphi$, and for any class $K$ contained in $V$ we have $F(\Phi_V(K)) = IS(K)$.
\end{theorem}
\begin{proof}
    Let $\varphi$ be a coherent condition on $V$, and consider the other coherent condition $\varphi'= \Phi_V(F(\varphi))$.
    Let $\theta \in \operatorname{Con}\mathbf{A}$.
    Notice, by Lemma \ref{lem:altfundclassdef}:
    \begin{align*}
        \theta \in \Spec^\varphi\mathbf{A} &\iff \mathbf{A}/\theta \in F(\varphi) =IS(F(\varphi))\\
        &\iff \theta \in \Spec^{\varphi'}\mathbf{A}
    \end{align*}
    Indeed, $\varphi = \varphi'$.\par
    Consider now a class $K$ contained in $V$.
    Clearly $IS(K) \subseteq F(\Phi_V(K))$, as obviously the identity map maps a member of $IS(K)$ into $IS(K)$ and has kernel $\Delta_A$.
    Conversely, suppose $\mathbf{A}\in F(\Phi_V(K))$
    $$\implies \mathbf{A} \cong \mathbf{A}/\Delta_A \in IS(K)$$
    Thus $F(\Phi_V(K)) \subseteq IS(K)$, and we are done.
\end{proof}

The universal algebraic geometry developed in \cite{plotkin2002sevenlecturesuniversalalgebraic}, can be reduced to looking at coherent conditions where the fundamental class is $IS(\mathbf{A})$.
We will see this in section 3.2.

\subsection{$\varphi$-radical congruences}

Theorem \ref{theo:coherency} suggests that the study of coherent conditions is equivalent to the study of classes closed under embeddings.
I.e., we can study a coherent condition by only looking at its fundamental class.\par
In this subsection we shall give an important example of this idea, in the form of $\varphi$-radical congruences.

\begin{definition}
    Let $\varphi$ be a coherent condition on $V$.
    We define the set of $\varphi$-radicals of an algebra $\mathbf{A}$, which we denote $\RSpec^\varphi\mathbf{A}$, to be the meet-completion of $\Spec^\varphi\mathbf{A}$.
    In other words, the $\varphi$-radical congruences of $\mathbf{A}$ are arbitrary intersections of $\varphi$-distinguished congruences.
\end{definition}

Because it is the meet completion, we can, similarly to prime spectra, define the $\varphi$-radical of some congruence $\theta$, and consequently the $\varphi$-nilradical of an algebra.

\begin{definition}
    Let $\varphi$ be a coherent condition on $V$, and $\mathbf{A} \in V$ and $\theta$ a congruence on $\mathbf{A}$.
    We define the following two congruences:
    $$\rad^\varphi\theta := \bigcap \{ \psi \in \Spec^\varphi\mathbf{A} \mid \theta \subseteq \psi\}$$
    $$\nil^\varphi\mathbf{A} := \rad^\varphi\Delta_A$$
    $\Delta_A$ being the minimal congruence.
    The quotient 
    $$\operatorname{red}^\varphi\mathbf{A} := \mathbf{A} / \nil^\varphi\mathbf{A} = \mathbf{A} / \bigcap\Spec^\varphi\mathbf{A}$$
    will be called the $\varphi$-\textit{reduction} of $\mathbf{A}$, and algebras where the nilradical is trivial will be called $\varphi$-\textit{reduced}
\end{definition}

We collect a couple of immediate facts from the definition.

\begin{proposition} \label{prop:radfacts}
    Let $\varphi$ be a coherent condition on $V$ and $\mathbf{A} \in V$.
    It follows that
    \begin{enumerate}
        \item For $\psi\in \Spec^\varphi\mathbf{A}$ and $\theta$ a congruence: $\theta \subseteq \psi \iff \rad^\varphi\theta \subseteq \psi$
        \item $\rad^\varphi$ is a closure operator on the congruence lattice of $\mathbf{A}$
        \item The set of closed elements is precisely $\RSpec^\varphi\mathbf{A}$
    \end{enumerate}
\end{proposition}
\begin{proof}
    $1)$ is a straightforward property of intersection.\par
    For $2)$, the extensive property is clear.
    To show that it is also increasing, consider $\theta_1 \subseteq \theta_2$.
    Thus for $\psi$ $\varphi$-distinguished, $\theta_2\subseteq \psi \implies \theta_1\subseteq \psi$, so
    $$\bigcap \{ \psi \in \Spec^\varphi\mathbf{A} \mid \theta_1 \subseteq \psi \} \subseteq \bigcap \{ \psi \in \Spec^\varphi\mathbf{A} \mid \theta_2 \subseteq \psi \}$$
    as desired.
    Idempotency and $3)$ will be done together, as $\rad^\varphi\theta \in \RSpec^\varphi$.
    Thus, consider an arbitrary element $\bigcap_{i \in I}\psi_i \in \RSpec^\varphi\mathbf{A}$ for $\psi_i \in \Spec^\varphi\mathbf{A}$.
    $$\implies \bigcap_{i \in I} \psi_i \subseteq \psi_j \; \forall j \in I$$
    $$\implies \rad^\varphi \bigcap_{i \in I} \psi_i \subseteq \bigcap_{i \in I} \psi_i$$
    But any congruence must also be contained in its $\varphi$-radical, so $\bigcap_{i \in I} \psi_i$ is its own $\varphi$-radical.
    Indeed, $\RSpec^\varphi\mathbf{A}$ is the set of $\varphi$-radical congruences, and $\rad^\varphi$ is idempotent, making it a closure operator.
\end{proof}

Element $1)$ in the proposition above can be seen as one direction of some generalised Nullstellensatz, and in the following subsection it will be fully presented as such.\par
Now we will focus on proving that the sets $\RSpec^\varphi\mathbf{A}$ define a coherent condition on $V$, and the rest of this subsection shall be devoted to looking at the properties of that coherent condition.
First we have a useful lemma.

\begin{lemma} \label{lem:radicalcommuteswithquot}
    Let $\varphi$ be a coherent condition on $V$, then for all congruences $\psi\subseteq\theta$ of $\mathbf{A} \in V$ we have the equality
    $$(\rad^\varphi\theta)/\psi = \rad^\varphi\theta/\psi$$
\end{lemma}
\begin{proof}
    Recall by Proposition \ref{prop:natprojisemb} that for a congruence $\psi$ of $\mathbf{A}$ the map
    $$\theta \mapsto \theta/\psi$$
    defines a correspondence between $\Spec^\varphi\mathbf{A}/\psi$ and the points in $\Spec^\varphi\mathbf{A}$ containing $\psi$.
    It is an elementary fact that this map respects intersections, so we see for any congruence $\theta \supseteq \psi$:
    \begin{align*}
        (\rad^\varphi\theta)/\psi &= (\bigcap \{ \gamma \in \Spec^\varphi\mathbf{A} \mid \theta \subseteq \gamma \}) / \psi\\
        &= \bigcap \{ \gamma/\psi \in \Spec^\varphi\mathbf{A}/\psi \mid \theta \subseteq \gamma \}\\
        &= \bigcap \{ \gamma/\psi \in \Spec^\varphi\mathbf{A}/\psi \mid \theta/\psi \subseteq \gamma/\psi \} = \rad^\varphi\theta/\psi
    \end{align*}
    the second to last equality making use of the correspondence theorem.
\end{proof}

A particular case of this result yields a nice formula for the $\varphi$-radical in terms of some nilradical.

\begin{corollary}
    $\rad^\varphi\theta = \pi_\theta^{-1}(\nil^\varphi\mathbf{A}/\theta)$,
    $\pi_\theta$ being the usual projection onto $\mathbf{A}/\theta$
\end{corollary}
\begin{proof}
    By Lemma \ref{lem:radicalcommuteswithquot}:
    $$\rad^\varphi\theta = \pi_\theta^{-1}(\rad^\varphi\theta/\theta) = \pi_\theta^{-1}(\nil^\varphi\mathbf{A}/\theta)$$
    as $\theta/\theta = \Delta_{A/\theta}$
\end{proof}

This is enough to prove the theorem.

\begin{theorem} \label{theo:radiscc}
    For any coherent condition $\varphi$ on $V$ there is a coherent condition $\sqrt{\varphi}$ on $V$ such that $\Spec^{\sqrt{\varphi}}\mathbf{A} = \RSpec^\varphi\mathbf{A}$.
\end{theorem}
\begin{proof}
    It suffices, of course, to show that the sets $\RSpec^\varphi\mathbf{A}$ follow the coherence properties.\par
    As such, let $h\colon \mathbf{A} \rightarrow \mathbf{B}$ be a homomorphism, and let $\theta \in \RSpec^\varphi\mathbf{B}$.
    By Proposition \ref{prop:radfacts}:
    \begin{align*}
        h^{-1}(\theta) &= h^{-1}(\bigcap\{ \psi \in \Spec^\varphi\mathbf{B} \mid \theta \subseteq \psi \})\\
        &= \bigcap \{ h^{-1}(\psi) \in \Spec^\varphi\mathbf{A} \mid \theta \subseteq \psi \})
    \end{align*}
    But, again as per Proposition \ref{prop:radfacts}, this must be an element of $\RSpec^\varphi\mathbf{A}$, so it satisfies the first coherence property.\par
    To show that it satisfies the second, consider $\theta \subseteq \psi$ congruences of $\mathbf{A}$, with $\psi$ being $\varphi$-radical.
    By Lemma \ref{lem:radicalcommuteswithquot}:
    $$\psi/\theta = (\rad^\varphi\psi)/\theta = \rad^\varphi\psi/\theta \in \RSpec^\varphi\mathbf{A}/\theta$$
    and we are done.
\end{proof}

From now on we shall refer to this coherent condition as $\sqrt{\varphi}$, called its \textit{radical}.
By Theorem \ref{theo:coherency} we therefore have that the class of $\varphi$-reduced algebras is precisely $F(\sqrt{\varphi})$, hence closed under embeddings, and also that $\theta$ is $\varphi$-radical if and only if $\mathbf{A}/\theta$ is reduced.
This also makes sense of the reduction of $\mathbf{A}$, as clearly $F(\sqrt{\varphi})$ is the class of algebras isomorphic to the reduction of some algebra in $V$.\par
It is natural to ask whether $F(\sqrt{\varphi})$ has anything to do with $F(\varphi)$ or more generally with $K$, in the case that $\varphi = \Phi_V(K)$.
We will answer this in the positive, and even show that the fundamental class simply encodes the property of a coherent condition to be the radical of some other.

\begin{theorem} \label{theo:radfundclassisprevar}
    Let $\varphi = \Phi_V(K)$ be a coherent condition.
    Then $F(\sqrt{\varphi}) = ISP(K)$.
    In other words, $\sqrt{\varphi} = \Phi_V(ISP(K))$, and so $\sqrt{\sqrt{\varphi}} = \sqrt{\varphi}$.
\end{theorem}
\begin{proof}
    We will show this by using a bi-inclusion.
    For the first direction, let $\mathbf{A} \in F(\sqrt{\varphi})$.
    Consider the natural map 
    $$\alpha \colon \mathbf{A} \rightarrow \prod_{\theta \in \Spec^\varphi\mathbf{A}}\mathbf{A}/\theta$$
    given by 
    $$a \mapsto (a/\theta)_{\theta \in \Spec^\varphi\mathbf{A}}$$
    This has kernel
    $$\bigcap\Spec^\varphi\mathbf{A} = \nil^\varphi\mathbf{A} = \Delta_A$$
    hence $\alpha$ is an embedding.
    Further, as $\mathbf{A}/\theta$ is in $IS(K)$ for all $\varphi$-distinguished $\theta$ we have
    $$\mathbf{A} \cong \operatorname{im}\alpha \leq \prod_{\theta \in \Spec^\varphi\mathbf{A}}\mathbf{A}/\theta \in ISP(K)$$
    $$\implies \mathbf{A} \in ISP(K)$$
    Hence $F(\sqrt{\varphi}) \subseteq ISP(K)$.\par
    Conversely, by Theorem \ref{theo:radiscc} it suffices to show that $F(\sqrt{\varphi})$ contains $P(K)$.
    Thus let $\mathbf{A}_i \in K$ be a family of algebras indexed by $I$, and define $\mathbf{A} = \prod_{i \in I}\mathbf{A}_i$.
    Let $\theta_i$ be the kernel of projection map $\pi_i \colon \mathbf{A} \rightarrow \mathbf{A}_i$.
    It should be clear that $\bigcap_{i \in I} \theta_i = \Delta_A$.
    But obviously $\theta_i \in \Spec^\varphi\mathbf{A}$, so
    $$\Delta_A = \bigcap_{i \in I} \theta_i \in \RSpec^\varphi\mathbf{A}$$
    Hence $P(K) \subseteq F(\sqrt{\varphi})$.\par
    That $\sqrt{\sqrt{\varphi}} = \sqrt{\varphi}$ is then translates to the fact that $ISP(ISP(K)) = ISP(K)$.
\end{proof}

\begin{corollary}
    $\varphi$ is the radical of some coherent condition if and only if $F(\varphi)$ is closed under arbitrary products if and only if $\Spec^\varphi\mathbf{A}$ is closed under arbitrary intersections for all $\mathbf{A}$.
\end{corollary}
\begin{proof}
    These follow directly from the above theorem.
\end{proof}

As mentioned in Section 1.1, prevarieties are precisely the models of some set of preidentities, so the radical coherent conditions are precisely those for which the fundamental class is defined by preidentities.

The most well-known (nonvariety) case of this is perhaps for commutative rings where $K$ is the class of fields, and $ISP(K)$ is axiomised by countable family of quasi-identities:
$$\left\{ x^n \approx 0 \rightarrow x\approx 0 \mid n \in \mathbb{N} \right\}$$

\subsection{Zariski closure and topology}

In this subsection we will introduce the Zariski closure operator and subsequent topology (by topologisation) on $\Spec^\varphi\mathbf{A}$.

\begin{definition}
    Let $\varphi$ be a coherent condition and $\mathbf{A} \in V$.
    For $S \subseteq A^2$ we define the \textit{Zariski closed set} generated by $S$:
    $$V_\varphi(S) := \{ \theta \in \Spec^\varphi\mathbf{A} \mid S\subseteq \theta\}$$
\end{definition}

It should be clear that $V_\varphi(S) = V_\varphi(\Theta(S))$ (where 
 $\Theta$ is the congruence closure).
Furthermore, as per Proposition \ref{prop:radfacts}, we have $V_\varphi(\theta) = V_\varphi(\rad^\varphi\theta)$.\par
These closed sets yield a closure system on $\Spec^\varphi\mathbf{A}$.

\begin{proposition}
    For any collection $V_\varphi(S_i)$ indexed by $I$, we have
    $$\bigcap_{i \in I} V_\varphi(S_i) = V_\varphi(\bigcup_{i \in I} S_i)$$
    So $\Spec^\varphi\mathbf{A}$ is a closure system.
\end{proposition}
\begin{proof}
    Let $\theta \in \Spec^\varphi\mathbf{A}$.
    By definition 
    $$\bigcup_{i \in I} S_i \subseteq \theta \iff S_i\subseteq \theta \; \forall i \in I$$
    So
    $$\theta \in V_\varphi(\bigcup_{i \in I} S_i) \iff \theta \in \bigcap_{i \in I} V_\varphi(S_i)$$
    hence the Zariski closed sets are closed under arbitrary intersection, and $\Spec^\varphi\mathbf{A}$ is a closure system.
\end{proof}

\begin{definition}
    A subset $X \subseteq \Spec^\varphi\mathbf{A}$ is \textit{Zariski topologically closed} if it is closed in the topologisation of the Zariski closure system, called the Zariski topology.\par
    We denote the Zariski closure of $X$ with $X^\text{zc}$ and the topological closure $\overline{X}$.
\end{definition}

\begin{corollary} \label{cor:principalclosedsets}
    The Zariski closed sets $V_\varphi(a_0, a_1)$ for $a_0, a_1 \in A$ (by abuse of notation) form a prebasis, or subbasis of the Zariski topology on $\Spec^\varphi\mathbf{A}$.
\end{corollary}
\begin{proof}
    Indeed, we have
    $$V_\varphi(S) = V_\varphi(\bigcup_{\langle a_0, a_1\rangle \in S} \langle a_0, a_1\rangle) = \bigcap_{\langle a_0, a_1\rangle \in S}V_\varphi(a_0, a_1)$$
    So the topology generated by taking the collection $V_\varphi(a_0, a_1)$ as prebasis includes every Zariski closed set, hence includes every Zariski topologically closed set, and so must be the Zariski topology.
\end{proof}

It is fruitful to make a distinction between Zariski closed- and topologically closed sets, as analogously to the usual algebraic geometry we have a notion of the Nullstellensatz.

\begin{definition}
    If $X \subseteq \Spec^\varphi\mathbf{A}$ we define
    $$\psi(X) := \bigcap X$$
    And for $S \subseteq A^2$ we say $\rad^\varphi S := \rad^\varphi \Theta(S)$
\end{definition}

\begin{theorem}
    [Generalised Nullstellensatz part I]
    Let $S \subseteq A^2$ and $X \subseteq \Spec^\varphi \mathbf{A}$, then
    $$S \subseteq \psi(X) \iff X \subseteq V_\varphi(S)$$
    So we have an antitone Galois connection.
    Furthermore:
    $$\psi(V_\varphi(S)) = \rad^\varphi S$$
    and
    $$V_\varphi(\psi(X)) = X^\text{zc}$$
\end{theorem}
\begin{proof}
    Let us begin with the antitone Galois connection.
    Suppose that $S \subseteq \psi(X)$, and let $\theta \in X$
    $$\implies S\subseteq \bigcap X \subseteq \theta$$
    $$\implies X \subseteq V_\varphi(S)$$
    For the reverse notice:
    $$X\subseteq V_\varphi(S) \implies S \subseteq \rad^\varphi S = \bigcap V_\varphi(S) \subseteq \bigcap X = \psi(X)$$
    As desired.\par
    The third and fourth item now follow from the fact that we have an antitone Galois connection.
\end{proof}

From this follows for example that $V_\varphi(-)$ and $\psi(-)$ are inclusion-reversing maps.
The following is a restatement of the last two items in the above theorem.

\begin{corollary}
    If $\varphi$ is a radical coherent condition, then $\psi(-)$ and $V_\varphi(-)$ give an antitone Galois-correspondence between $\varphi$-distinguished congruences and Zariski closed sets.
\end{corollary}

Following classical (universal) algebraic geometry, we can interpret $\mathbf{A}/\psi(X)$ as the `coordinate algebra' of the set of `points' $X$.
If $\varphi = \Phi_V(K)$, then Theorem \ref{theo:radfundclassisprevar} tells us that that class of (algebras isomorphic to) coordinate algebras of $\varphi$ is precisely $ISP(K)$.\par
The assignment $\mathbf{A} \mapsto \Spec^\varphi\mathbf{A}$ happens to be a contravariant functor to $\mathbf{Top}$.

\begin{proposition}
    A homomorphism $f \colon \mathbf{A} \rightarrow \mathbf{B}$ induces a closure morphism $f_\ast \colon \Spec^\varphi\mathbf{B} \rightarrow\Spec^\varphi\mathbf{A}$ given by the preimage.
    This yields a contravariant functor from $V$ to $\mathbf{Top}$.
\end{proposition}
\begin{proof}
    By the fundamental properties of coherent conditions we obviously have a function $f_\ast \colon \Spec^\varphi\mathbf{B} \rightarrow\Spec^\varphi\mathbf{A}$.
    To see that this is also a closure morphism, take any Zariski closed set $X = V_\varphi(\theta) \subseteq \Spec^\varphi\mathbf{A}$.
    Consider 
    $$f^{-1}_\ast(X) = \{ \psi \in \Spec^\varphi\mathbf{B} \mid \theta \subseteq f^{-1}(\psi) \}$$
    and define $f(\alpha) = \{ \langle f(a), f(b)\rangle \mid\langle a, b\rangle \in \alpha \}$.
    Notice that $f(f^{-1}(\psi)) = \psi$, so
    $$\theta  \subseteq f^{-1}(\psi) \implies f(\theta) \subseteq \psi$$
    Then
    $$f(\theta) \subseteq \psi \implies \theta \subseteq f^{-1}(f(\theta)) \subseteq f^{-1}(\psi)$$
    Hence 
    $$f^{-1}_\ast(X) =\{ \psi \in \Spec^\varphi\mathbf{B} \mid f(\theta) \subseteq \psi \} = V_\varphi(f(\theta))$$
    which is Zariski closed.
    By Proposition \ref{prop:topisendfunc} this is then a continuous function between the Zariski topologies.\par
    That this yields a contravariant functor follows from elementary properties of the preimage.
\end{proof}

A Zariski closed set $X \subseteq \Spec^\varphi\mathbf{A}$ turns out to have a very nice description given the subspace topology, as again the spectrum of some algebra.

\begin{theorem} \label{theo:closedsetisspec}
    Let $X \subseteq \Spec^\varphi\mathbf{A}$ be an arbitrary subset.
    Then
    $$X^\text{zc} \cong \Spec^\varphi \mathbf{A}/\psi(X)$$
    The homeomorphism being given by the natural projection $\pi \colon \mathbf{A} \rightarrow \mathbf{A}/\psi(X)$
\end{theorem}
\begin{proof}
    By Proposition \ref{prop:natprojisemb}, $\pi_\ast \colon \Spec^\varphi\mathbf{A} \rightarrow \Spec^\varphi\mathbf{A}/\psi(X)$ is an embedding with image $V_\varphi(\psi(X))$, which is just $X^\text{zc}$ by the generalised Nullstellensatz part I.\par
    Consider the set-theoretic inverse (this is the inverse of $\pi_\ast$ by the correspondence theorem):
    $$\pi_\ast^{-1} \colon \theta \mapsto \theta/\psi(X)$$
    This is also a morphism of closure systems, which can be seen by noticing the fact that the preimage of $V_\varphi(\theta/\psi(X))$ under $\pi_\ast^{-1}$ is simply $V_\varphi(\theta)$.
    As such $X^\text{zc} \cong \Spec^\varphi\mathbf{A}/\psi(X)$, since an isomorphism of closure operators induces a homeomorphism on the topologisations.
\end{proof}

\begin{corollary}
    $\Spec^\varphi\mathbf{A} \cong\Spec^\varphi(\operatorname{red}^\varphi\mathbf{A})$
\end{corollary}

As such, we see that coherent conditions make no difference at all between an algebra and its reduction.

\section{$K$-Spectra of free algebras}

During this section we fix a variety $V$ of type $\mathscr{F}$ and a class $K$ contained in $V$, which generates the coherent condition $\varphi = \Phi_V(K)$.\par
It is a fairly elementary fact that every algebra $\mathbf{A}\in V$ is isomorphic to some quotient of $\mathbf{F}_V(X)$ for some set $X$, for example by taking $X$ to be $A$, or some other generating set of $\mathbf{A}$.
Theorem \ref{theo:radfundclassisprevar} therefore suggest that, to study the $K$-spectra ($\varphi$-spectra) in $V$, it suffices to study simply the Zariski (topologically) closed sets of $\mathbf{F}_V(X)$.
That is what we will do now, consequently showing the connections between the various types of closed sets, and sets of formulas, and stating the second part of the universal Nullstellensatz.\par
The approach in the first subsection is inspired by \cite{Equationaldomains}.

\subsection{Disjunctive systems}

Recall that a subset $T \subseteq F_V(X)^2$ can just as well be treated as a system of equations for algebras in $V$, i.e. as a shorthand for the proposition $\bigwedge_{\langle p, q \rangle \in T} p \approx q$.
Hence $\mathbf{A} \models T(\vec{a})$ means every equation in $T$ is satisfied by $\vec{a} \in A^X$ in $\mathbf{A}$.
Also, for $\vec{a} \in A^X$ we denote $h_{\vec{a}} \colon \mathbf{F}_V(X) \rightarrow \mathbf{A}$ for the evaluation map induced by the choice of variables $x \mapsto a_x$.

This gives us a nice interpretation of the Zariski closed sets of $\Spec^\varphi\mathbf{F}_V(X)$.

\begin{lemma} \label{lem:closedsetisevalmapker}
    Let $T \subseteq F_V(X)^2$ be a system of equations.
    It follows that
    $$V_\varphi(T) = \{ \theta \in \operatorname{Con}\mathbf{F}_V(X) \mid \theta = \ker h_{\vec{a}} \text{ where } h_{\vec{a}} \colon \mathbf{F}_V(X) \to \mathbf{A} \in K, \mathbf{A} \models T(\vec{a})\}$$
\end{lemma}
\begin{proof}
    The kernel of any $h_{\vec{a}}$ to $\mathbf{A} \in K$ is by definition $\varphi$-distinguished, and by the universal mapping property any homomorphism is an evaluation morphism, so any $\varphi$-distinguished congruence must be the kernel of an evaluation morphism into $K$.\par
    Now:
    \begin{align*}
        \mathbf{A} \models T(\vec{a}) &\iff p^\mathbf{A}(\vec{a}) = q^\mathbf{A}(\vec{a}) \; \forall \langle p, q\rangle \in T\\
        &\iff h_{\vec{a}}(p) = h_{\vec{a}}(q) \; \forall \langle p, q\rangle \in T\\
        &\iff T\subseteq \ker h_{\vec{a}}
    \end{align*}
    so we are done.
\end{proof}

We can generalise this to an arbitrary algebra $\mathbf{B}$ as follows:
There are always exists some projection $\pi \colon \mathbf{F}_V(X) \rightarrow \mathbf{B}$, and $\theta$ of $\mathbf{B}$ is $\varphi$-distinguished iff $\pi^{-1}(\theta)$ is.
As such, $V_\varphi(S)\subseteq \Spec^\varphi\mathbf{B}$ can be thought of as the (kernels of) interpretations of $\mathbf{B}$ in some $\mathbf{A} \in K$ that satisfy the relations in $S$.\par
Let us now begin working towards the second part of the Nullstellensatz.
There happen to be nice 'closed forms` for nonempty Zariski closed and Zariski topologically closed sets of the $K$-spectrum of $\mathbf{F}_V(X)$.

\begin{lemma} \label{lem:closedformZariskitop}
    For pairs $\epsilon_{i, k} \in F_V(X)^2$ where $n_i>0$:
    $$\bigcap_{i \in I} V_\varphi(\epsilon_{i, 1}) \cup\dots\cup V_\varphi(\epsilon_{i, n_i}) = \{ \theta \in \Spec^\varphi\mathbf{F}_V(X) \mid K\models \theta \rightarrow \bigwedge_{i \in I}\bigvee_{k=1}^{n_i} \epsilon_{i, k}\}$$
    And, in particular for a system of equations $T$:
    $$V_\varphi(T) = \{ \theta \in \Spec^\varphi \mathbf{F}_V(X) \mid K \models \theta \rightarrow T \}$$
\end{lemma}
\begin{proof}
    We will start by proving the claim that 
    $$V_\varphi(\epsilon_1) \cup \dots \cup V_\varphi(\epsilon_n) = \{ \theta \in \Spec^\varphi \mathbf{F}_V(X) \mid K \models \theta \to \epsilon_1 \vee \dots \vee \epsilon_n \}$$
    One direction of the inclusion is fairly direct.
    If $\theta \in V_\varphi(\epsilon_k)$ for some $1 \leq k \leq n$, then $\epsilon_k \in \theta$, so of course $K \models \theta \rightarrow \epsilon_k$, and in particular $K \models \theta \rightarrow \epsilon_1 \vee \dots \vee \epsilon_k$.\par
    Conversely, suppose that $K \models \theta \to \epsilon_1 \vee \dots \vee \epsilon_n$ for some congruence $\theta$.
    Then
    $$\mathbf{A} := \mathbf{F}_V(X) / \theta \in IS(K)$$
    $$\implies \mathbf{A} \models \theta \to \epsilon_1 \vee \dots \vee \epsilon_n$$
    Let $\bar{x} \in A^X$ be the sequence of elements $(x/\theta)_{x \in X}$.
    Notice then that $\pi_\theta \colon \mathbf{F}_V(X) \to \mathbf{A}$ is simply the evaluation map $h_{\bar{x}}$, using the notation of the lemma above.
    So we see, for $\langle p, q \rangle \in \theta$:
    $$p^\mathbf{A}(\bar{x}) = \pi_\theta(p) = \pi_\theta(q) = q^\mathbf{A}(\bar{x})$$
    $$\implies \mathbf{A} \models \theta(\bar{x})$$
    hence there exists a $1 \leq k \leq n$ with $\mathbf{A} \models \epsilon_k(\bar{x})$, and therefore by Lemma \ref{lem:closedsetisevalmapker} we have that $\theta \in V_\varphi(\epsilon_k) \subseteq V_\varphi(\epsilon_1) \cup \dots \cup V_\varphi(\epsilon_n)$, as desired.\par
    The rest of the proof now follows trivially.
\end{proof}

This motivates the following definition.

\begin{definition}
    Let $S$ be a set of finite disjunctions of equations (or pairs in $F_V(X)$), which we call a \textit{distjunctive system}.
    We define
    $$V_\varphi(S) = \{ \theta \in  \Spec^\varphi\mathbf{F}_V(X) \mid K\models \theta \rightarrow \bigwedge_{\Phi \in S}\Phi \}$$
    In particular, this notation coincides with the usual one whenever $S$ consists only of equations.
\end{definition}

From Corollary \ref{cor:principalclosedsets} we know that the principal closed sets $V_\varphi(\epsilon)$ give a prebasis of the Zariski topology.
What Lemma \ref{lem:closedformZariskitop} then tells us, is that every nonempty Zariski topologically closed set of $\mathbf{F}_V(X)$ can be written as $V_\varphi(S)$ for some disjunctive system $S$.

\begin{lemma} \label{lem:monotonicity}
    Let $S_1, S_2$ be distjunctive systems with variables in $X$.
    Then
    $$V_\varphi(S_1) \subseteq V_\varphi(S_2) \iff K \models \bigwedge_{\Phi_1  \in S_1} \Phi_1 \rightarrow \bigwedge_{\Phi_2 \in S_2} \Phi_2$$
\end{lemma}
\begin{proof}
    $(\Leftarrow)$
    This direction is fairly clear.
    Suppose that $K \models \bigwedge_{\Phi_1  \in S_1} \Phi_1 \rightarrow \bigwedge_{\Phi_2 \in S_2} \Phi_2$, and let $\theta\in V_\varphi(S_1)$
    $$\implies K\models \theta \rightarrow \bigwedge_{\Phi_1 \in  S_1} \Phi_1$$
    $$\implies K \models \theta \rightarrow \bigwedge_{\Phi_2\in S_2}\Phi_2$$
    $$\implies \theta \in V_\varphi(S_2)$$
    simply by transitivity.\par
    $(\Rightarrow)$
    To see that the converse is also true, we again suppose that $V_\varphi(S_1)\subseteq V_\varphi(S_2)$.
    Now, consider $\mathbf{A} \in K$ and $\vec{a} \in A^X$ such that
    $$\mathbf{A} \models (\bigwedge_{\Phi_1\in S_1}\Phi_1)(\vec{a})$$
    Thus, for all $\bigvee_{k=1}^n p_k\approx q_k \in S_1$ there exists a $1\leq i \leq n$ such that $p_i^\mathbf{A}(\vec{a}) = q^\mathbf{A}_i(\vec{a})$.
    By Lemma \ref{lem:closedsetisevalmapker}:
    $$\ker h_{\vec{a}} \in V_\varphi(\langle p_i, q_i\rangle) \subseteq V_\varphi(\bigvee_{k=1}^n p_k\approx q_k)$$
    $$\implies \ker h_{\vec{a}} \in \bigcap_{\Phi \in S_1} V_\varphi(\Phi) = V_\varphi(S_1)$$
    $$\implies \ker h_{\vec{a}} \in V_\varphi(S_2)$$
    One can then do the same dance to show  that $\mathbf{A} \models (\bigwedge_{\Phi_2 \in S_2} \Phi_2)(\vec{a})$.
    Therefore
    $$\mathbf{A}\models \bigwedge_{\Phi_1  \in S_1} \Phi_1 \rightarrow \bigwedge_{\Phi_2 \in S_2} \Phi_2$$
    And, as $\mathbf{A}$ was arbitrary, $K$ must satisfy that formula too.
\end{proof}

We now state the second part of the generalised Nullstellensatz.

\begin{theorem} [Generalised Nullstellensatz part II]
    Let $S_1, S_2$ be disjunctive systems over $X$ and $A \subseteq \Spec^\varphi\mathbf{F}_V(X)$ Zariski topologically closed and nonempty.
    Then
    $$V_\varphi(S_1) \subseteq V_\varphi(S_2) \iff K \models \bigwedge_{\Phi_1  \in S_1} \Phi_1 \rightarrow \bigwedge_{\Phi_2 \in S_2} \Phi_2$$
    Furthermore, for $T_1, T_2 \subseteq F_V(X)^2$:
    $$V_\varphi(T_1) = \{ \theta  \in \Spec^\varphi\mathbf{F}_V(X) \mid K\models \theta\rightarrow T_1\}$$
    $$T_1\subseteq \rad^\varphi T_2  \iff K\models T_2 \rightarrow T_1$$
\end{theorem}
\begin{proof}
    The first two are simply Lemmas \ref{lem:closedformZariskitop} and \ref{lem:monotonicity}.
    The third uses the fact that $T \subseteq F_V(X)^2$ can be interpreted as a set of atomic formulas $S$, and as such
    $$K \models T \leftrightarrow \bigwedge_{\epsilon \in S} \epsilon$$
    and thus by the generalised Nullstellensatz part I and Lemma \ref{lem:monotonicity}
    $$\rad^\varphi T_1 \subseteq \rad^\varphi T_2 \iff K  \models T_2 \rightarrow T_1$$
    Therefore, by elementary properties of closure operators:
    $$T_1 \subseteq \rad^\varphi T_2 \iff \rad^\varphi T_1 \subseteq \rad^\varphi T_2  \iff K \models T_2 \rightarrow T_1$$
    as desired.
\end{proof}

The case $\varnothing \subseteq \Spec^\varphi\mathbf{F}_V(X)$ is worth discussing.
Closed sets of the form $V_\varphi(S)$ with $S$ a disjunctive system are precisely the sets which can be written as an intersection of finite unions of Zariski closed sets (as the principal sets $V_\varphi(p, q)$ form a subbasis).
This means that $V_\varphi(S) = \varnothing$ for some $S$ iff $\varnothing$ is Zariski closed, iff the trivial algebra is not contained in $IS(K)$.

Also, the third item of the generalised Nullstellensatz can be restated as
$$\rad^\varphi T = \{ \langle p, q\rangle \in F_V(X)^2 \mid K\models T \rightarrow p\approx q\}$$
Here we can really see that congruences play the role of points in $A^ n$ for some algebra $\mathbf{A}$: $p^\mathbf{A}(a_1, \dots, a_n) = q^\mathbf{A}(a_1, \dots, a_n)$ for $(a_1, \dots, a_n) \in A^n$ is simply replaced by $K\models \theta \rightarrow p\approx q$.
Indeed, the analogy between these two is the essence of Lemma \ref{lem:closedsetisevalmapker}.

\subsection{Connections with standard universal algebraic geometry}

We will follow the definitions and notation of \cite{daniyarova2011algebraicgeometryalgebraicstructures}.
Fix an algebra $\mathbf{A}$ over $\mathscr{F}$ and let $\mathfrak{F}$ denote the class of all algebras of type $\mathscr{F}$.
This subsection will aim to show that the coherent condition $\varphi_\mathbf{A} := \Phi_{\mathfrak{F}}(\mathbf{A})$ is functionally the same as the standard universal algebraic geometry of $\mathbf{A}$.

\begin{definition}
    Let $T \subseteq \operatorname{At}(x_1, \dots, x_n)$ be a set of equations.
    Then 
    $$V_\mathbf{A}(T) = \{ (a_1, \dots, a_n) \mid p^\mathbf{A}(a_1, \dots, a_n) = q^\mathbf{A}(a_1, \dots, a_n) \; \forall p\approx q \in T \}$$
    And if $X \subseteq A^n$ then
    $$\operatorname{Rad}(X)= \{ p\approx q \in \operatorname{At}(x_1, \dots, x_n) \mid p^\mathbf{A}(a_1, \dots, a_n) = q^\mathbf{A}(a_1, \dots, a_n) \; \forall (a_1, \dots, a_n) \in X \}$$
    $\theta_{\operatorname{Rad}(X)}$ denotes the congruence on the term algebra $\mathbf{T}(x_1, \dots, x_n)$ corresponding to $\operatorname{Rad}(X)$.
\end{definition}

One can easily check that the algebraic sets $V_\mathbf{A}(T)$ define a closure system on $A^n$, which we (uniquely to this paper and mainly for convenience) denote by $\mathbb{A}^n(\mathbf{A})$.
We aim to show that $\mathbb{A}^n(\mathbf{A})$ is actually in many ways no different from $\Spec^\varphi \mathbf{T}(x_1,\dots, x_n)$, by way of a quasi-isomorphism.\par
Recall that, for $\vec{a} = (a_1, \dots, a_n)$, $h_{\vec{a}} \colon \mathbf{T}(x_1, \dots, x_n) \rightarrow \mathbf{A}$ denotes the homomorphism taking $x_i$ to $a_i$.

\begin{theorem}
    The map
    $$\alpha \colon A^n \rightarrow \Spec^\varphi \mathbf{T}(x_1, \dots, x_n)$$
    $$(a_1, \dots, a_n) \mapsto \ker h_{\vec{a}}$$
    is a quasi-isomorphism from $\mathbb{A}^n(\mathbf{A})$ to the $\varphi$-spectrum of the term algebra.
\end{theorem}
\begin{proof}
    It should be obvious that $\alpha$ is a well-defined map into $\Spec^\varphi\mathbf{T}(x_1, \dots, x_n)$.
    Let $V_\varphi(S) \subseteq \Spec^\varphi\mathbf{T}(x_1, \dots, x_n)$ be a Zariski closed set, and note that we can interpret $S$ as a system of equations.
    Consider
    $$\alpha^{-1}(V_\varphi(S)) = \{ (a_1, \dots, a_n) \mid S \subseteq \ker h_{\vec{a}} \}$$
    But of course
    \begin{align*}
        S \subseteq \ker h_{\vec{a}} &\iff h_{\vec{a}}(p) = h_{\vec{a}}(q) \; \forall \langle p, q\rangle \in S\\
        &\iff p^\mathbf{A}(\vec{a}) = q^\mathbf{A}(\vec{a})\\
        &\iff \vec{a} \in V_\mathbf{A}(S)
    \end{align*}
    $$\implies \alpha^{-1}(V_\varphi(S)) = V_\mathbf{A}(S)$$
    So $\alpha$ is a morphism of closure systems.
    Now suppose $V_\mathbf{A}(T) \subseteq A^n$ an algebraic set.
    Then by the above $\alpha^{-1}(V_\varphi(T)) = V_\mathbf{A}(T)$, so $\alpha(V_\mathbf{A}(T)) = V_\varphi(T)$, as $\alpha$ is surjective.\par
    This proves that $\alpha$ sends closed sets to closed sets, and also that $\alpha^{-1}$ gives a bijection between closed sets, so we get the desired result.
\end{proof}

As a quasi-isomorphism induces an isomorphism on the lattices of closed sets, properties like chain conditions are equivalent between quasi-isomorphic closure systems, but also this means that the Zariski closure operator on $\Spec^\varphi \mathbf{T}(x_1, \dots, x_n)$ is topological iff the one on $\mathbb{A}^n(\mathbf{A})$ is.
Furthermore, as we will see, quasi-isomorphisms respect and reflect irreducibility.\par
This last fact is obvious.

\begin{proposition}
    $\theta_{\operatorname{Rad}(X)} = \psi(\alpha(X))$ for any $X\subseteq k^n$
\end{proposition}

\section{Equationally Noetherian classes}

Unique to universal algebraic geometry (and essentially any algebraic geometry that is not that of commutative rings) is the need for a notion of an equationally Noetherian algebra.
Recalling from \cite{daniyarova2011algebraicgeometryalgebraicstructures}, we have the following definition:

\begin{definition}
    An algebra $\mathbf{A}$ is equationally Noetherian iff for any $T\subseteq \operatorname{At}(x_1, \dots, x_n)$ there exists a finite subsystem $T_0 \subseteq T$ with $V_\mathbf{A}(T) = V_\mathbf{A}(T_0)$ (in the sense of section 3.2) or equivalently $\mathbf{A} \models T \leftrightarrow T_0$
\end{definition}

From (\cite{Daniyarova2008UnificationTI}, Remark 4.8) we have the following:

\begin{proposition} \label{prop:noethtopology}
    Let $X$ be a topological space with prebasis $\mathcal{B}$ of closed subsets.
    Then the following are equivalent:
    \begin{itemize}
        \item The closed sets of $X$ satisfy the descending chain condition;
        \item $\mathcal{B}$ satisfies the descending chain condition.
    \end{itemize}
    And if either of these are satisfied, in which case we call $X$ Noetherian, we have
    \begin{itemize}
        \item $X$ has a finite number of irreducible components
        \item Every closed subset of $X$ can be written as a finite union of irreducible sets.
        In particular, $X$ is covered by its irreducible components.
    \end{itemize}
\end{proposition}

Recal that a subset $Y \subseteq X$ is irreducible if, for any closed $U, V \subseteq X$, $Y \subseteq U \cup V$ implies that $Y \subseteq U$ or $Y \subseteq V$, and is an irreducible component if it is irreducible and maximal with that property.

Taking Definition 4.1 above, we can extend the notion of equationally Noetherian to arbitrary classes.
We fix a class contained in a variety $K \subseteq V$ and let $\varphi = \Phi_V(K)$.

\begin{definition}
    $K$ is equationally Noetherian if for any $T \subseteq \operatorname{At}(x_1, \dots, x_n)$ there is a finite subset $T_0 \subseteq T$ we have $K \models T\leftrightarrow T_0$, or equivalently $K\models T_0\rightarrow T$.
    We say that $\varphi$ is Noetherian if $K$ is equationally Noetherian.
\end{definition}

There are a couple elementary equivalent conditions of equational Noetherianness.

\begin{theorem} \label{theo:eqnoethequivcond}
    The following are equivalent:
    \begin{enumerate}
        \item $K$ is equationally Noetherian
        \item $\Spec^\varphi\mathbf{F}_V(x_1, \dots, x_n)$ satisfies the D.C.C. (descending chain condition) on Zariski closed sets for all $n$ or equivalently $A.C.C.$ (ascending chain condition) on the $\varphi$-radical congruences.
        \item $\Spec^\varphi\mathbf{F}_V(x_1, \dots, x_n)$ satisfies the D.C.C. on Zariski topologically closed sets for all $n$
        \item For any disjunctive system $S$ over a finite set of variables $X$ there is a finite subset $S_0 \subseteq S$ with $K\models S\leftrightarrow S_0$
        \item Condition 2. for any finitely generated algebra in $V$
        \item Condition 3. for any finitely generated algebra in $V$
    \end{enumerate}
\end{theorem}
\begin{proof}
    $(1\Rightarrow 2)$
    Suppose that $K$ is equationally Noetherian, and consider then an ascending chain of $\varphi$-radical congruences
    $$\theta_0 \subseteq \theta_1 \subseteq \theta_2 \subseteq \dots $$
    and let $\theta = \rad^\varphi (\bigcup_{i \in \mathbb{N}} \theta_i)$ be their join in the lattice of $\varphi$-radical congruences.
    By the hypothesis and the generalised Nullstellensatz part II, there is a finite $T_0 \subseteq \bigcup_{i \in \mathbb{N}} \theta_i$ such that $\rad^\varphi T_0 = \theta$.
    But, as $T_0$ is finite, there must be some $n$ such that $T_0 \subseteq \theta_n$, so $T_0 \subseteq \theta_k$ for all $k\geq n$.
    $$\implies \rad^\varphi T_0 \subseteq \rad^\varphi \theta_k = \theta_k \subseteq \theta = \rad^\varphi T_0$$
    Indeed, the chain stabilises.\par
    $(3\Rightarrow 4)$ 
    Now, suppose that $\Spec^\varphi \mathbf{F}_V(x_1, \dots, x_n)$ satisfies the D.C.C. on Zariski topologically closed sets and let $S$ be a disjunctive system over $\{ x_1, \dots, x_n\}$.
    Consider the set of all $V_\varphi(S_*)$ for $S_* \subseteq S$ finite.
    By the descending chain condition, this set has a minimal element $V_\varphi(S_0)$.
    Now supposing any $\Phi \in S$, we have
    $$\implies V_\varphi(S_0) \cap V_\varphi(\{\Phi\}) \subseteq V_\varphi(S_0)$$
    But $V_\varphi(S_0)$ is minimal, so in fact $V_\varphi(S_0) \cap V_\varphi(\{\Phi\}) = V_\varphi(S_0)$
    $$\implies V_\varphi(S_0) \subseteq V_\varphi(\{ \Phi\})$$
    $$\implies K\models \bigwedge_{\Phi_0 \in S_0} \Phi_0 \rightarrow \Phi$$
    Thus $K\models \bigwedge_{\Phi_0 \in S_0} \Phi_0 \rightarrow \bigwedge_{\Phi \in S} \Phi$, and we are done.\par
    $(2 \Rightarrow 3)$ follows from the above proposition, and $(4\Rightarrow 1)$ is immediate.
    Finally, the equivalences $2\iff 5$ and $3\iff 6$ follow from Theorem \ref{theo:closedsetisspec}.
\end{proof}

\begin{corollary}
    $\varphi$ being Noetherian implies that $\sqrt{\varphi}$ is Noetherian.
\end{corollary}
\begin{proof}
    By the generalised Nullstellensatz part I, there is a Galois correspondence between $\varphi$-radical congruences and Zariski closed sets of $\Spec^\varphi\mathbf{A}$.
    Then, there is also a Galois correspondence between $\varphi$-radical congruences and Zariski closed sets of $\Spec^{\sqrt{\varphi}} \mathbf{A}$.\par
    We now use the above theorem to get the desired result.
\end{proof}

\begin{corollary}
    If $K$ is equationally Noetherian, then for any finitely generated algebra $\mathbf{A}$, $\Spec^\varphi\mathbf{A}$ is hereditarily compact (every subspace is compact).
\end{corollary}
\begin{proof}
    If $K$ is equationally Noetherian, then by the above theorem the topological space $\Spec^\varphi\mathbf{A}$ is Noetherian (satisfies the A.C.C. on open sets) for any finitely generated $\mathbf{A}$.
    Then, any subspace of a Noetherian space is Noetherian, and any Noetherian space is compact, from which we get that $\Spec^\varphi\mathbf{A}$ is hereditarily compact.
\end{proof}

\begin{remark}
    In general, $\Spec^\varphi \mathbf{A}$ is compact whenever $\nabla_A$ is a compact element in the lattice of $\varphi$-radical congruences.
\end{remark}

Theorem \ref{theo:eqnoethequivcond} shows that we can define equational Noetherianness for coherent conditions without having to refer to or even look at its fundamental class or any class generating it.
For example, if $V$ is a Noetherian variety (one where any finitely generated algebra has finitely generated congruences) then $K$ is immediately also equationally Noetherian.
In particular, Hilbert's basis theorem says that $\mathbb{Z}[x_1, \dots, x_n]$ is Noetherian, so the variety of commutative rings is equationally Noetherian, and therefore every algebraically closed field is, which gives us the result from classical algebraic geometry.\par
Next, we prove a partial generalisation of Unification Theorem C from \cite{daniyarova2011algebraicgeometryalgebraicstructures}.

\begin{theorem} \label{theo:fingenISPisQ}
    For an equationally Noetherian class $K$ and $\mathbf{A} \in V$ finitely generated, and $\varphi = \Phi_V(K)$, the following are equivalent:
    \begin{enumerate}
        \item $\mathbf{A} \in Q(K)$, the quasivariety generated by $K$
        \item $\mathbf{A} \in ISP(K)$
        \item $\mathbf{A}$ is $\varphi$-reduced
        \item $\mathbf{A}$ is separated by $K$
    \end{enumerate}
\end{theorem}
\begin{proof}
    The proofs of $(2) \iff (4)$ and $(2) \implies (1)$ are exhibited in Lemma 3.5 of \cite{Daniyarova2008UnificationTI}, and $(2)\iff (3)$ is simply Theorem \ref{theo:radfundclassisprevar}.\par
    All there is left to show is $(1) \implies (3)$.
    Suppose thus that $\mathbf{A} \in Q(K)$, and consider a surjection $h_{\vec{a}} \colon \mathbf{F}_V(x_1, \dots, x_n) \rightarrow \mathbf{A}$ sending $x_i$ to $a_i$, which exists as $\mathbf{A}$ is finitely generated.\par
    Suppose that $\langle p, q\rangle \in \rad^\varphi(\ker h_{\vec{a}})$
    $$\implies K \models \ker h_{\vec{a}} \rightarrow p\approx q$$
    But, as $K$ is equationally Noetherian, there is a finite subset $S_0 \subseteq \ker h$ such that $K \models S_0 \leftrightarrow \ker h_{\vec{a}}$
    $$\implies K \models S_0 \rightarrow p\approx q$$
    This is now a quasi-identity satisfied by $K$, and therefore also by $\mathbf{A}$.
    Of course, we have $\mathbf{A} \models S_0(\vec{a})$, and thus $p^\mathbf{A}(\vec{a}) = q^\mathbf{A}(\vec{a})$, which means that $\langle p, q \rangle \in \ker h_{\vec{a}}$.
    In other words: $\rad^\varphi(\ker h_{\vec{a}}) = \ker h_{\vec{a}}$, making $\mathbf{A}$ reduced as desired.
\end{proof}

We let $L_\omega(K)$ be the \textit{local closure} of $K$; the class of algebras where every finitely generated subalgebra is a member of $K$.
As $Q(K)$ is axiomised by finitary propositions, $L_\omega(Q(K)) = Q(K)$.
The following corollary is then immediate.

\begin{corollary}
    Let $K$ be an equationally Noetherian class.
    Then
    \begin{enumerate}
        \item $L_\omega(K)$ is equationally Noetherian
        \item $Q(K) = L_\omega(ISP(K))$
        \item $Q(K)$ is equationally Noetherian
    \end{enumerate}
\end{corollary}
\begin{proof}
    Note for 1. that, for any finitely generated algebra $\mathbf{A}$, the spectra of the coherent conditions generated by $K$ and $L_\omega(K)$ coincide.\par
    For the second, it is clear that $L_\omega(ISP(K)) \subseteq Q(K)$, as $Q(K)$ is locally closed.
    Conversely, if $\mathbf{A} \in Q(K)$, then by Theorem \ref{theo:fingenISPisQ} any finitely generated algebra of $\mathbf{A}$ is in $ISP(K)$, so $\mathbf{A} \in L_\omega(ISP(K))$.
\end{proof}

\section{Irreducible closed sets}

Again, throughout this section, let $V$ be a fixed variety, $K$ a class of algebras and $\varphi = \Phi_V(K)$ the corresponding coherent condition.

\subsection{Irreducible sets}

We begin with a high-level overview of irreducibility in arbitrary closure operators.
The usual definition translates naturally into that setting.

\begin{definition}
    Let $\langle X, C\rangle$ be a closure system.
    An \textit{irreducible subset} of $X$ is a nonempty subset $Y\subseteq X$ such that
    $$Y \subseteq \bigcup_{k=1}^n A_n \implies \exists k : Y\subseteq A_k$$
    for any closed sets $A_k$.
\end{definition}

The following proposition is obvious.

\begin{proposition}
    $Y\subseteq X$ is irreducible implies $C(Y)$ is irreducible.
\end{proposition}

For topological spaces, there is an even weaker condition equivalent to irreducibility of a subset, relating to prebases.
This is useful for us as the Zariski topology is a topologization, and topologizations are defined using prebases.

\begin{proposition} \label{lem:topspaceirr}
    Let $X$ be a topological space.
    Then $Y \subseteq X$ is irreducible for a prebasis of closed sets $\mathcal{B}$ (the irreducibility condition but only for sets in the prebasis instead of any closed set) iff it is irreducible in $X$.
\end{proposition}
\begin{proof}
    The backwards direction is obvious.\par
    $(\Rightarrow)$
    Suppose $Y \subseteq X$ is irreducible for a prebasis in $\mathcal{B}$, and
    $$Y \subseteq A \cup B$$
    for $A, B$ closed in $X$.
    $$\implies Y \subseteq (\bigcap_{i \in I} A_{i, 1} \cup \dots \cup A_{i, k_i}) \cup (\bigcap_{j \in J} B_{j, 1} \cup \dots \cup B_{j, l_j})$$
    for $A_{i, k}, B_{j, l}\in \mathcal{B}$
    $$\implies Y \subseteq \bigcap_{(i, j) \in I\times J} A_{i, 1} \cup \dots \cup A_{i, k_i} \cup B_{j, 1} \cup \dots \cup B_{j, l_j}$$
    Thus for any $i \in I$, $j \in J$
    $$Y \subseteq A_{i, 1} \cup \dots \cup A_{i, k_i} \cup B_{j, 1} \cup \dots \cup B_{j, l_j}$$
    If for all $i \in I$ there is an $1\leq x_i \leq k_i$ with $Y \subseteq A_{i, x_i}$, which means that $Y \subseteq A$ and we are done.
    In the other case, let $i \in I$ where this is not the case.
    By irreducibility in $X$ this means for all $j \in J$ there is an $1\leq y_j \leq l_j$ with $Y \subseteq B_{j, l_j}$.
    But this means that $Y \subseteq B$, so we are done.
\end{proof}

As the principal closed sets form a prebasis of the Zariski topology, we get for an arbitrary coherent condition $\varphi$ a nice corollary.

\begin{corollary} \label{col:irrsetzar}
    A subset $X \subseteq \Spec^\varphi\mathbf{A}$ is irreducible in the Zariski topology if and only if for all $a_i, b_i \in A$, $i = 1, \dots, n$:
    $$X \subseteq \bigcup_{i=1}^nV_\varphi(a_i, b_i) \implies \exists k \leq n : X \subseteq V_\varphi(a_k, b_k)$$
    if and only if it is irreducible in the Zariski closure system if and only if $\overline{X}$ and $X^{\operatorname{zc}}$ are irreducible.
\end{corollary}

It follows that maximal irreducible sets in the Zariski topology are Zariski closed, as
$$X\subseteq \overline{X} \subseteq X^{\text{zc}}$$
We shall now show that the study of irreducible Zariski closed sets can be reduced to the study of irreducible spectra.

\begin{proposition}
    If $Y \subseteq X$, then $Y$ is irreducible in $X$ if and only if $Y$ is in the corresponding subclosure system $\hat{Y}$.
\end{proposition}
\begin{proof}
    We note that the closed sets of $\hat{Y}$ are precisely $Y\cap A$ where $A$ is closed in $X$.
    Then the desired result follows from the fact that
    $$Y \subseteq Y \cap S \iff Y = Y\cap S \iff Y\subseteq S$$
    for any $S \subseteq X$, and the fact that $Y\cap (A \cup B) = (Y\cap A) \cup (Y\cap B)$.
\end{proof}

By Theorem \ref{theo:closedsetisspec} it then follows that irreducible Zariski closed sets indeed correspond to irreducible spectra of reduced algebras.

\begin{proposition}
    Let $f \colon X \rightarrow Y$ be a quasi-isomorphism of closure systems.
    In particular, the quasi-isomorphism $\alpha \colon \mathbb{A}^k(A) \rightarrow \Spec^{\Phi_\mathfrak{F}(\mathbf{A})} \mathbf{T}(x_1, \dots, x_n)$.
    It is then the case that a closed set $A \subseteq X$ is irreducible iff $f(A)$ is irreducible.
\end{proposition}
\begin{proof}
    The proof trivially follows from the fact that the direct and preimage give a bijection between the closed sets, and that they both preserve unions.
\end{proof}

Lastly we return to Noetherian spaces and give an extension of Proposition \ref{prop:noethtopology}.

\begin{proposition} \label{prop:uniqueirrdec}
    Let $X$ be a Noetherian topological space (it satisfies the DCC on closed sets), then it has a unique (up to ordering) minimal decomposition into irreducibles $X = C_1 \cup \dots \cup C_n$, where $C_1, \dots, C_n$ are the irreducible components.
\end{proposition}
\begin{proof}
    By Proposition \ref{prop:noethtopology} $X$ can be written as the union $C_1 \cup \dots \cup C_n$ of irreducible components.
    First we will show maximality of this decomposition.
    Suppose WLOG that $C_n \subseteq C_1 \cup \dots \cup C_{n-1}$.
    Then, by irreducibility, $C_n \subseteq C_k$ for some $k$, which is obviously a contradiction as these are maximal irreducible subsets, so the decomposition is minimal.\par
    Consider now another minimal decomposition $X = A_1 \cup \dots \cup A_m$ into irreducibles. Then for any irreducible component $C_k$ we have $C_k \subseteq A_1 \cup \dots \cup A_m$, and so $C_k = A_l$ for some $l$, by maximality and irreducibility of $C_k$.
    Thus $m \geq n$, and every $C_k$ appears in the given decomposition. 
    This decomposition is also minimal, however, so every $A_l$ must be some $C_k$ as the irreducible components already cover $X$.
    This proves the proposition.
\end{proof}

\subsection{Irreducible spectra of $\varphi$-reduced algebras}

Corollary \ref{col:irrsetzar} gave a definition for a closed in terms of principal Zariski closed sets.
It will be apparent, though, that this is not necessarily useful to connect the irreducibility of the $\varphi$-spectrum to properties of the $\varphi$-reduction.
For this we need the distinguished open sets.

\begin{definition}
    Let $a_0, a_1 \in A$.
    Then the \textit{distinguished open set} of $(a_i, b_i)$ is the set
    $$D^\varphi(a_i, b_i) := \{ \theta \in \Spec^\varphi\mathbf{A} \mid \langle a_0, a_1\rangle \notin \theta \}$$
\end{definition}

Its clear that $D^\varphi(a_i, b_i) = \Spec^\varphi\mathbf{A} - V_\varphi(a_0, a_1)$, and that it is the set of kernels of homomorphisms into $K$ which separate $a_0$ and $a_1$.
We can then reframe the definition of an irreducible spectrum into a very convenient one with distinguished open sets.

\begin{lemma} \label{lem:irreduciblereducedequivcond}
    $\Spec^\varphi\mathbf{A}$ is irreducible and $\mathbf{A}$ is $\varphi$-reduced if and only if any finite intersection of distinguished open sets $D_\varphi(a_i, b_i)$ for $a_i\neq b_i$ is nonempty.
\end{lemma}
\begin{proof}
    First note that
    \begin{align*}
        \bigcap_{i=1}^n D^\varphi(a_i, b_i) = \varnothing &\iff \Spec^\varphi \mathbf{A} - \bigcap_{i=1}^n D^\varphi(a_i, b_i) = \Spec^\varphi\mathbf{A}\\
        &\iff \bigcup_{i=1}^n (\Spec^\varphi\mathbf{A} - D^\varphi(a_i, b_i)) = \Spec^\varphi\mathbf{A}\\
        &\iff \bigcup_{i=1}^n V_\varphi(a_i, b_i) = \Spec^\varphi\mathbf{A}
    \end{align*}
    $(\Rightarrow)$
    Suppose that $\Spec^\varphi\mathbf{A}$ is irreducible and $\mathbf{A}$ is $\varphi$-reduced, and by way of contradiction that $\bigcap_{i=1}^n D^\varphi(a_i, b_i) = \varnothing$ for $a_i\neq b_i$.
    $$\implies \bigcup_{i=1}^n V_\varphi(a_i, b_i) = \Spec^\varphi\mathbf{A}$$
    $$\implies \exists k \leq n : \Spec^\varphi\mathbf{A} \subseteq V_\varphi(a_k, b_k)$$
    $$\implies \langle a_k, b_k\rangle \in \nil^\varphi\mathbf{A} = \Delta_A$$
    But this is a contradiction, as $a_i \neq b_i$ for all $i$, so the original intersection must be empty.\par
    $(\Leftarrow)$
    Now suppose that any finite intersection of distinguished subsets $D^\varphi(a_i, b_i)$ with $a_i\neq b_i$ is nonempty.
    This, in particular, means that $D^\varphi(a, b)$ is nonempty if $a\neq b$, and therefore $\nil^\varphi\mathbf{A} = \Delta_A$.
    Then one can then repeat the argument above to see that this also implies irreducibility.
\end{proof}

This is a really powerful characterisation.
The intersection of $D^\varphi(a_i, b_i)$ for $a_i\neq b_i$ always being nonempty is equivalent to the property that, for any finite family of pairs $\langle a_i, b_i\rangle$, there exists a homomorphism into $K$ separating all of them, i.e., that the algebra is $n$-separated by $K$.

\begin{theorem}
    $\mathbf{A}$ is $n$-separated by $K$ if and only if it is $\varphi$-reduced and has an irreducible spectrum.
\end{theorem}

To connect this to algebras discriminated by $K$, we notice the following fact:

\begin{lemma} \label{prop:disisgensep}
    $K$ discriminates $\mathbf{A}$ if and only if $K$ $n$-separates $\mathbf{A}$ for all $n$.
\end{lemma}
\begin{proof}
    Indeed, if $K$ discriminates $\mathbf{A}$ and $a_i\neq b_i \in A$ for $1\leq i\leq n$, then there must be a homomorphism pairwise separating $W = \{ a_1, b_1, \dots, a_n, b_n\}$.\par
    Conversely, if $W = \{ a_1, \dots, a_n\} \subseteq A$ then we get the pairs
    $a_i \neq a_j \; \forall i\neq j$
    for which there must be a homomorphism into $K$ separating them.
    This means it pairwise separates $W$, and is therefore injective when restricted to $W$.
\end{proof}

And we can immediately prove the theorem.

\begin{theorem} \label{theo:irrsetisprime}
    The following are equivalent:
    \begin{enumerate}
        \item $\mathbf{A}$ is $\varphi$-reduced and has an irreducible $\varphi$-spectrum
        \item For any $a_i\neq b_i$, we have $\bigcap_{i=1}^n D^\varphi(a_i, b_i) \neq \varnothing$
        \item $\mathbf{A} \in \mathbf{Sep}_\omega(K)$
        \item $\mathbf{A} \in \mathbf{Dis}(K)$
    \end{enumerate}
\end{theorem}
\begin{proof}
    $(1) \iff (2)$ are by Lemma \ref{lem:irreduciblereducedequivcond}, $(2) \iff (3)$ by the discussion above, and $(3) \iff (4)$ is by Lemma \ref{prop:disisgensep}.
\end{proof}

From here it is easy to see that the $\varphi$-distinguished congruences $\theta$ for which $\mathbf{A}/\theta$ has an irreducible spectrum form another coherent condition, as clearly $\mathbf{Sep}_\omega(K)$ is closed under embeddings.
We call these congruences $\varphi$-prime.

We end this section by giving a result like the one in classical algebraic geometry.

\begin{theorem}
    If $\mathbf{A}$ is $C$-Noetherian, or $\varphi$ is equationally Noetherian and $\mathbf{A}$ is finitely generated, then every $\varphi$-radical congruence has a decomposition into a finite number of $\varphi$-prime congruences.
\end{theorem}
\begin{proof}
    Let $\mathbf{A}$ be an algebra satisfying any of the two hyptheses above for some coherent condition $\varphi$.
    Note that quotients of $C$-Noetherian or finitely generated algebras are respectively $C$-Noetherian and finitely generated, so it is enough to consider the case of $\mathbf{A}$ being reduced and finding a decomposition of $\Delta_A$.\par
    Then by Theorem \ref{theo:eqnoethequivcond}.6, $\Spec^\varphi\mathbf{A}$ is a Noetherian topological space.
    So, let $\Spec^\varphi \mathbf{A} = C_1 \cup \dots \cup C_n$ be the decomposition into irreducible components which is minimal by Proposition \ref{prop:noethtopology}.
    Then define $\psi_i := \psi(C_i)$, which by Theorem \ref{theo:irrsetisprime} is a prime congruence.
    Notice that
    $$\psi_i \supset \psi_1 \cap \dots \cap \psi_n \; \forall i$$
    $$\implies C_i \subseteq V_\varphi(\psi_1 \cap \dots \cap \psi_n)\; \forall i$$
    Hence $V_\varphi(\psi_1 \cap \dots \cap \psi_n) = \Spec^\varphi\mathbf{A}$ and therefore $\psi_1 \cap \dots \cap \psi_n = \Delta_A$.
\end{proof}

Note that we cannot expect this given decomposition to be unique, or even minimal, as $V_\varphi(\psi_1 \cap \psi_2)$ might not be $V_\varphi(\psi_1) \cup V_\varphi(\psi_2)$.

\printbibliography

\end{document}